\documentclass[10pt,reqno]{amsart}
\usepackage[utf8]{inputenc}
\usepackage[UKenglish]{babel}
\usepackage{amsmath, amssymb, amsthm}
\usepackage{mathrsfs}
\usepackage{enumitem}
\usepackage{graphicx}
\usepackage{tikz-cd}
\usepackage{xcolor}
\numberwithin{equation}{section} 

\usepackage{diagbox} 
\usepackage{tensor}  

\usepackage{hyperref} 

\usepackage[alphabetic]{amsrefs} 

\theoremstyle{plain}
\newtheorem{thm}{Theorem}[section]
\newtheorem{prop}[thm]{Proposition}
\newtheorem{lemma}[thm]{Lemma}

\theoremstyle{definition}
\newtheorem{defin}[thm]{Definition}

\newtheorem{rmk}[thm]{Remark}

\allowdisplaybreaks

\newcommand{\R}{\mathbb{R}}

\newcommand{\MP}{\mathcal{MP}}
\newcommand{\A}{\mathbb{A}}

\def\XXint#1#2#3{{\setbox0=\hbox{$#1{#2#3}{\int}$ }
\vcenter{\hbox{$#2#3$ }}\kern-.6\wd0}}

\usepackage{scalerel,stackengine}
\stackMath
\newcommand\hhat[1]{%
\savestack{\tmpbox}{\stretchto{%
  \scaleto{%
    \scalerel*[\widthof{\ensuremath{#1}}]{\kern.1pt\mathchar"0362\kern.1pt}%
    {\rule{0ex}{\textheight}}
  }{\textheight}%
}{2.4ex}}%
\stackon[-6.9pt]{#1}{\tmpbox}%
}

\title[Boundary and scattering rigidity for MP-systems]{The boundary and scattering rigidity problems for simple MP-systems}
\author{Sebastián Muñoz-Thon}

\address{Department of Mathematics, Purdue University, West Lafayette, IN 47907.}
\email{smunozth@purdue.edu}

\begin{document}

\begin{abstract}
We consider a compact Riemannian manifold with
boundary, endowed with a magnetic field and a potential. This is called an $\MP$-system. On simple $\MP$-systems, we consider both the boundary rigidity problem and the scattering rigidity problem. Under the assumption of knowing the boundary action functions or the scattering relations for only one energy level, we prove that we can recover the $\MP$-system up to a gauge. To obtain the results, we reduce the $\MP$-system to the corresponding magnetic system and we apply the results from \cite{dpsu} on simple magnetic systems.
\end{abstract}

\maketitle

\section{Introduction} 

\subsection{Known and new results} \label{subsec:known} 

The boundary rigidity problem is a classical question in inverse problems. Given a smooth ($C^{\infty}$) compact Riemannian manifold with smooth boundary, the problem asks to what extent one can recover the metric $g$ by knowing the boundary distance function $d_{g}|_{\partial M \times \partial M}$ (\cite{michel}). Assuming that $(M,g)$ is a simple surface (see \cite{psu}*{Section 3.8} for a several equivalent definitions of simple manifolds), the problem was solved in \cite{pu2005}, proving that one can recover the metric up to an isometry that fixes the boundary. Recently, the problem was solved in higher dimensions for compact simple manifolds with further geometric conditions (non-positive curvature, non-negative curvature, absence of focal points), and for manifolds with strictly convex boundary satisfying a foliation condition, see \cite{suv}. There are results generic results \cite{su}, and for non-convex manifolds \cite{gmt}.

In a more general setting, one can consider the same problem for \emph{magnetic systems}, that is, for a smooth compact Riemannian manifold with smooth boundary $(M,g)$ endowed with a closed 2-form $\Omega$. These systems model the motion of a unit charge of unit mass in a magnetic field, but they also appear in other contexts, such as, differential and symplectic geometry, and dynamical systems, see \cite{dpsu} and the references there. The boundary rigidity problem in the presence of a magnetic field asks to what extent one can recover the metric $g$ and the magnetic field $\Omega$ by knowing the boundary action function $\mathbb{A}_{g,\Omega}|_{\partial M \times \partial M}$ (see \cite{dpsu} for the definition of $\mathbb{A}$). This problem was solved (up to a natural gauge) in \cite{dpsu} (under a simplicity assumption), on a conformal class, for analytic magnetic systems, for surfaces, and for metric close to a generic set of metrics. It is also worth mentioning that similar problem have also been studied. The \emph{scattering rigidity problem} for magnetic system was studied also in \cite{dpsu}, \cite{herrerossca} and \cite{herrerosmag}, while the \emph{lens rigidity problem} was studied in \cite{zhou18}.

The next step is to consider the boundary rigidity problem for $\MP$-systems. An $\MP$-system consist of a smooth compact Riemannian manifold with smooth boundary $(M,g)$, a closed 2-form $\Omega$, and a smooth function $U$. The acronym $\MP$ was first (to the author knowledge) used in \cite{az}, and stands for magnetic and potential. Likewise, this acronym comes from the fact that these kind of systems works with curves that describe the motion of a particle on a Riemannian manifold under the influence of a magnetic field represented by $\Omega$, and a potential field represented by the function $U$. These kind of systems appear in mechanics, see \cite{kozlov}, \cite{arnold}, \cite{akn}, \cite{cggmw}, \cite{maraner}. They also appear when one studies geodesics on Lorentzian manifolds endowed with stationary metrics \cite{germinario}, \cite{bg}, \cite{plamen}, in inverse problems for the acoustic wave equation from phaseless measurements \cite{iw}, and in inverse problems in transport equations with external forces on Euclidean domains \cite{lz}.

Given an $\MP$-system $(M,g,\Omega,U)$, the magnetic field $\Omega$ induces a map $Y \colon TM \to TM$ given by
\[ \Omega_{x}(u,v)=(Y_{x}u,v)_{g}, \]
where $u,v \in T_{x}M$. This map is called \emph{Lorentz force} associated to the magnetic field $\Omega$. For simple $\MP$ systems (see Section \ref{sec:prelim} below), $C^{2}$ curves $\sigma \colon [a,b] \to M$ that satisfy
\begin{equation} \label{eq:mp-geo}
    \nabla_{\dot{\sigma}} \dot{\sigma}=Y(\dot{\sigma})-\nabla U(\sigma),
\end{equation}
are called \emph{$\MP$-geodesics}. Here, $\nabla$ is the Levi-Civita connection associated to the metric $g$. Equation \eqref{eq:mp-geo} defines a flow, called the \emph{$\MP$-flow}, and given by
\[ \phi_{t}(x,v)=(\sigma(t),\dot{\sigma}(t)),  \]
where $\sigma$ solves \eqref{eq:mp-geo} and $\sigma(0)=x$, $\dot{\sigma}(0)=v$. See Lemma \ref{lemma:mp_flow} for other interpretations of the $\MP$-flow. For this flow, the energy $E \colon TM \to \R$ given by $E(x,v)=\frac{1}{2}|v|_{g}^{2}+U(x)$ is an integral of motion. Indeed, for $\sigma$ satisfying \eqref{eq:mp-geo}, we have
\[ \frac{d}{dt}E(\sigma(t),\dot{\sigma}(t))=( \nabla_{\dot{\sigma}}\dot{\sigma},\dot{\sigma})_{g}+(\nabla U,\dot{\sigma})=( Y \dot{\sigma},\dot{\sigma})_{g}=\Omega(\dot{\sigma},\dot{\sigma})=0. \]
Then the energy is constant along $\MP$-geodesics. If $\Omega=d\alpha$, it has been shown that $\MP$ geodesics minimize the \emph{time free action} of energy $k$ (see \cite{az}*{Appendix A.1})
\[ \A(\sigma)=\frac{1}{2} \int_{0}^{T}|\dot{\sigma}|_{g}^{2}dt+kT-\int_{0}^{T}(\alpha(\sigma(t),\dot{\sigma}(t))+U(\sigma(t)) )dt, \]
where 
\[ \sigma \in \mathcal{C}(x,y)=\{\sigma \in AC([0,T],M): \sigma(0)=x \text{ and } \sigma(T)=y\}, \]
so that the \emph{Ma\~n\'e action potential} (of energy $k$) is well defined
\begin{equation} \label{eq:mane_act}
    \A(x,y)=\inf_{\gamma \in \mathcal{C}(x,y)}\A(\gamma).
\end{equation}
The ``action'' terminology comes form physics and Lagrangian flows, see \cite{am}, \cite{arnoldclas}, \cite{ci}, \cite{paternain}, \cite{mazzucchelli}.

The \emph{boundary rigidity problem for $\MP$ systems} asks to what extent one can recover the metric $g$, the magnetic field $\Omega$ (or $\alpha$), and the potential $U$, by knowing the boundary action function $\mathbb{A}_{g,\alpha,U}|_{\partial M \times \partial M}$ (of energy $k$). In the same spirit, \emph{scattering rigidity problem for $\MP$ systems} asks to what extent one can recover $(g,\Omega,U)$ by knowing the scattering relation (of energy $k$), see Definition \ref{defin:scattering}. In the flat case, these problems were studied in \cite{jo2007}. The non-flat simple case was studied in \cite{az}. There, the authors prove that for simple $\MP$-systems, the knowledge of the boundary action function for two energy levels allows to recover the system $(g,\alpha,U)$ up to a gauge, in three cases: working on the same conformal class, for analytic $\MP$-systems, and working on surfaces. Furthermore, they give a counterexample that shows that, under their gauge, the result is false if one only knows one energy level. 

In this work, we propose a new notion of gauge for $\MP$ system (see Definition \ref{defin:gauge}). For simple $\MP$-systems (see Definition \ref{defin:simple}), we are able to solve the boundary rigidity problem by the knowledge of the boundary action function for one energy level, up to the new gauge, in the following cases:
\begin{itemize}
    \item in a fixed conformal class (Theorem \ref{thm:bound_rid_conf_class});
    \item for real-analytic $\MP$-systems (Theorem \ref{thm:bound_rid_an});
    \item on surfaces (Theorem \ref{thm:bound_rid_dim2}).
\end{itemize}
The idea of the proof is the following. We reduce the simple $\MP$-system of energy $k$ to the simple magnetic system $(2(k-U)g,\alpha)$, then we use the results in \cite{dpsu}, and finally we ``pull back'' the results to the $\MP$-system. Of course, in the two last steps is where the gauge appears. As is mentioned in \cite{az}, the boundary rigidity problem and the scattering rigidity problem are equivalent for simple manifolds. Hence, our results also solve the scattering rigidity problem. These results are stated at the end of Section \ref{sec:thms}.

\subsection{Structure of the paper}

In Section \ref{sec:prelim} we briefly summarize some facts about $\MP$-systems and they relation with magnetic systems that would be used in this work. In Section \ref{sec:hamiltonians} we compare different Hamiltonians that appear in the study of $\MP$-systems. This serves as a motivation for the introduction of the new gauge (Definition \ref{defin:gauge}) in Section \ref{sec:defin}. In Section \ref{sec:thms} we prove the boundary and scattering rigidity results mentioned at the end of Subsection \ref{subsec:known}. Finally, in Section \ref{sec:rmk} we give some remarks and possible future directions of research, and in the appendix we give characterizations of the $\MP$-flow from Lagrangian and Hamiltonian point of view.

\subsection*{Acknowledgments}

The author would like to thank Plamen Stefanov for suggesting this problem, for helpful discussions about magnetic systems, and for helpful comments on a previous version of this manuscript. Also, the author would like to thank Gabriel Paternain and Gunther Uhlmann for helpful suggestions on a preliminary version for this work. The author was partly supported by NSF Grant DMS-2154489.


\section{Preliminaries} \label{sec:prelim}

In this section we review some results for $\MP$-systems that will be used throughout the paper. 

Recall that an $\MP$-system is consist of tuple $(M,g,\Omega,U)$, where $M$ is a compact smooth manifold with smooth boundary, $g$ is a Riemannian metric, $\Omega$ is a closed 2-form and $U$ is a smooth function. Since we are going to work on a fixed manifold $M$, we will going to refer to the triple $(g,\Omega,U)$ as an $\MP$-system. 

As was mentioned in Subsection \ref{subsec:known}, $\MP$-geodesics has constant energy. Given $k \in \R$, we define $S^{k}M:=\{E=k\}$. We will assume always that $k>\max_{M} U$ (so that $S^{k}M$ is a non-empty level set, see also \cite{az}*{Appendix A} and \cite{ci}*{Chapter 2}). Let $\nu(x)$ be the inward unit vector normal to $\partial M$ at $x$, and set
\[
\partial_{\pm} S^{k}M:=\{(x, v) \in S^{k}M: x \in \partial M, \pm (v, \nu(x))_{g(x)} \geq 0\}.
\]
For $x \in M$, the $\MP$-exponential map at $x$ is the map $\exp _x^{\MP} \colon T_{x} M \to M$ given by
\[
\exp _x^{\MP}(t v)=\pi \circ \phi_t(v),
\]
where $t \geq 0$, $v \in S_{x}^{k} M$, and $\pi \colon TM \to M$ is the canonical projection. 

Let $\Lambda$ denotes the second fundamental form of $\partial M$. Consider a manifold $M_1$ such that $M_1^{\text {int }} \supset M$. Extend $g$, $\Omega$ and $U$ to $M_1$ smoothly, preserving the former notation for extensions. $M$ is said to be \emph{$\MP$-convex} at $x \in \partial M$ if there is a neighborhood $O$ of $x$ in $M_1$ such that all $\MP$-geodesics of constant energy $k$ in $O$, passing through $x$ and tangent to $\partial M$ at $x$, lie in $M_1 \setminus M^{\text {int }}$. If, in addition, these geodesics do not intersect $M$ except for $x$, we say that $M$ is \emph{strictly $\MP$-convex} at $x$. By \cite{az}*{Lemma~A.2}, strictly $\MP$-convexity implies
\[
\Lambda(x, v)>\langle Y_x(v), \nu(x)\rangle-d_x U(\nu(x))
\]
for all $(x, v) \in S^{k}(\partial M)$. 

\begin{defin} \label{defin:simple}
We say that $M$ is ($\MP$) \emph{simple} with respect to $(g, \Omega, U)$ if $\partial M$ is strictly $\MP$-convex and the $\MP$-exponential map $\exp_{x}^{\MP} \colon (\exp_x^{\MP})^{-1}(M) \to M$ is a diffeomorphism for every $x \in M$.    
\end{defin}

In this case, $M$ is diffeomorphic to the unit ball of $\R^n$. Hence, Poincaré's lemma implies that $\Omega$ is exact, that is, there exist a 1-form $\alpha$ such that $\Omega=d\alpha$. and we call $\alpha$ to be the magnetic potential.

Henceforth we call $(g, \alpha, U)$ a simple $\MP$-system on $M$. We will also say that $(M, g, \alpha, U)$ is a simple $\MP$-system. When $\alpha=0$ and $U=0$, these notions coincide with the usual notion of simple Riemannian manifold, see \cite{psu}*{Section~3.8}.

For $(x, v) \in \partial_{+} S^{k}M$, let $\tau(p, v)$ be the time exit time function for the the $\MP$-geodesic $\sigma$ with $\sigma(0)=p$, $\dot{\sigma}(0)=v$. By \cite{az}*{Lemma~A.3} we have that for a simple $\MP$-system, the function $\tau \colon \partial_{+} S^{k}M \to \R$ is smooth.

\begin{defin} \label{defin:scattering}
The \emph{scattering relation} $\mathcal{S} \colon \partial_{+} S^{k}M \to \partial_{-} S^{k}M$ of an $\MP$-system $(M, g, \alpha, U)$ is defined as
\[
\mathcal{S}(p,v)=(\phi_{\tau (p, v)}(p, v))=(\sigma(\tau(p,v)), \dot{\sigma}(\tau(p,v))).
\]    
\end{defin}

Note that the scattering relation depends on the energy level.

\begin{defin}
Given a simple $\MP$ system $(g,\alpha,U)$ of energy $k$, we associate to it the magnetic system $(2(k-U)g,\alpha)$ of energy $\frac{1}{2}$, which we call \emph{reduced magnetic system}.    
\end{defin}

The interplay between an $\MP$-system an its magnetic reduction was studied in detail in \cite{az}. We now state some results that are useful for this work.

\begin{lemma}[\cite{az}*{Proposition~1, 2, 3}] \label{lemma:basics_mp}
Let $(g,\alpha,U)$ be $\MP$-system with energy $k$ an let $(G,\alpha):=(2(k-U)g,\alpha)$ be its reduction to a magnetic system of energy $\frac{1}{2}$.
\begin{enumerate}
    \item If $k>\max_{x \in M}U$ and $\sigma$ is a $\MP$-geodesic of energy $k$, then there exists a reparametrization of $\sigma$ that is a unit speed magnetic geodesic for the reduced system.
    \item $(g,\alpha,U)$ is simple (in the $\MP$ sense) if and only if $(G,\alpha)$ is simple (in the magnetic sense).
    \item Let $\A$ be the Ma\~n\'e's action potential of energy $k$ for $(g, \alpha, U)$ and $\mathbb{A}_G$ be the Ma\~n\'e's action potential of energy 1/2 for the simple magnetic system $(G, \alpha)$, then $\A|_{\partial M \times \partial M}=\A_{G}|_{\partial M \times \partial M}$.
\end{enumerate}
\end{lemma}

The definition of simple magnetic systems and the definition of the Ma\~n\'e's potential for magnetic systems, are the same as above with $U=0$.

The new parameter such that $\gamma(s)=\sigma(t(s))$ is a magnetic unit speed geodesic is given explicitly by
\[ s(t)=\int_{0}^{t}2(k-U(\sigma))dt. \]
Part (1) on Lemma \ref{lemma:basics_mp} is known as Jacobi--Maupertuis' principle (cf. \cite{am}*{Theorem 3.7.7}).

\section{Hamiltonians} \label{sec:hamiltonians}

In this section we study some relations that appear when one is working with different hamiltonians on $\MP$ and magnetic systems.

Associated to the simple $\MP$ system $(g,\alpha,U)$ on $M$, we have the Hamiltonian $H \colon T^{*}M \to \R$ and the energy $E \colon TM \to \R$ given by
\[ H(x,\xi)=\frac{1}{2}|\xi+\alpha|_{g(x)}^{2}+U(x), \qquad E(x,v)=\frac{1}{2}|v|_{g(x)}^{2}+U(x),\]
respectively. We will fix $k>\max_{x \in M}U$, and consider $\{H=k\} \subset T^{*}M$ and $S^{k}M=\{E=k\} \subset TM$. Knowing the values of the Ma\~n\'e's action restricted to the boundary on $\MP$-geodesics within that energy level, the question is if we can recover the triple $(g,\alpha,U)$ up to some gauge. We begin by considering a more general Hamiltonian $\tilde{H} \colon T^{*}M \to \R$ given by
\[ \tilde{H}(x,\xi)=\tilde{H}_{\mu,k}(x,\xi):=\frac{\mu(x)}{2}|\xi+\alpha|_{g(x)}^{2}+\mu(x)(U(x)-k)+k. \]
Here, $\mu$ is a smooth positive function on $M$, usually called \emph{elliptic factor}. We also consider the associated energy $\tilde{E} \colon TM \to \R$ given by
\[ \tilde{E}(x,v)=\tilde{E}_{\mu,k}(x,v)=\frac{1}{2\mu(x)}|v|_{g(x)}^{2}+\mu(x)(U(x)-k)+k. \]
Note that $\tilde{H}$ and $\Tilde{E}$ are related by the Legendre transform $(x,v) \mapsto (x,\frac{1}{\mu}v^{\flat}-\alpha)$ (see also Lemma \ref{lemma:mp_flow}). 
First, let us justify why we chose this Hamiltonian. The motivation is that we want to find a more general Hamiltonian with the same energy level of the original one, and with the same $\MP$-geodesics, up to reparametrization. To find it, note that the function $\Tilde{H}$ arises when given the condition $H=k$, one looks up for a Hamiltonian with the same level of energy. In that way, one could take the constant $k$ to the left-hand side, to obtain $H-k=0$. Motivated by the fact that multiplying a Hamiltonian by an elliptic factor only reparametrizes the zero level curves, we multiply by the elliptic factor $\mu$ on both sides to obtain $\mu(x)(H-k)=0$. Since we want the same energy level as before, we add $k$ to both sides, to obtain $\tilde{H}_{\mu,k}=\mu (H-k)+k=k$. 

On the other hand, from the equation $k=H$, one could take the potential from the right-hand side to the left-hand side, to obtain
\[ k-U(x)=\frac{1}{2}|\xi+\alpha|_{g(x)}^{2}. \]
From this we obtain
\[ 1=\frac{1}{2(k-U(x))}|\xi+\alpha|_{g(x)}^{2}. \]
Since we want the same energy level as before, we need to multiply both sides by $k$. Hence, we find
\[ k=\frac{k}{2(k-U(x))}|\xi+\alpha|_{g(x)}^{2}=:\hat{H}(x,\xi). \]
The problem with this last Hamiltonian is that we lose the generality of the elliptic factor that we had before. Furthermore, we have that $\hat{H}_{k}=\tilde{H}_{\mu,k}$ with $\mu=k/(k-U)$, and $\{\hat{H}_{k}=k\}=\{\tilde{H}_{\mu,k}=k\}$. Hence, the Hamiltonian $\Tilde{H}_{\mu,k}$ is more general that $\hat{H}_{k}$.

On the other hand, one could ask to use the Hamiltonian of the reduced magnetic system $(2(k-U)g,\alpha)$ associated to the $\MP$-system $(g,\alpha,U)$. The Hamiltonian that corresponds to the system $(2(k-U)g,U)$ is
\[ H_{r}:=\frac{1}{2(k-U)} \frac{|\xi+\alpha|_{g}^{2}}{2}. \]
In \cite{az}, the authors work with this Hamiltonian on $\{H_{r}=1/2\}$. Observe that if we take $\mu=1/(2(k-U))$, we obtain $\{H_{r}=1/2\}=\{\tilde{H}_{\mu,k}=k\}$, and $H_{r}=\tilde{H}_{\mu,k}-(k-1/2)$. 

So far, we have that $\tilde{H}_{\mu,k}$ is a generalization of $H$, $\hat{H}_{k}$, and $H_{r}$. The following table summarizes the information up to this point.
\begin{center}
    \begin{tabular}{|c|c|c|c|} \hline
  $\Tilde{H}_{\mu,k}$ & $H=\Tilde{H}_{1,k}$ & $\hat{H}_{k}=\tilde{H}_{\frac{k}{k-U},k}$ & $H_{r}=\tilde{H}_{\frac{1}{2(k-U)},k}-\left(k-\frac{1}{2}\right)$\\\hline
$\frac{1}{\mu}g$ & $g$ & $\frac{k-U}{k}g$ & $2(k-U)g$ \\ \hline
$\alpha$ & $\alpha$ & $\alpha$ & $\alpha$ \\ \hline
$\mu(U-k)+k$ & $U$ & $0$ & $0$ \\ \hline
$k$ & $k$ & $k$ & $\frac{1}{2}$ \\ \hline
\end{tabular}
\end{center}
Here, in the first row we write the Hamiltonians, in the second row we write the metric, in the third row the magnetic field, in the fourth row the potential, and in the last row the energy level.

\begin{rmk} \label{rmk:red}
As we mentioned above, the authors in \cite{az} reduce the $\MP$-system $(g,\alpha,U)$ (with fixed energy level $k$) to the magnetic system $(G,\alpha)$, where $G=2(k-U)g$. If we do the same procedure whit the $\MP$-system $(g/\mu,\alpha,\mu(U-k)+k)$, we obtain  again the magnetic system $(G,\alpha)$.    
\end{rmk}

Before focusing on the study of $\tilde{H}$, we note the following relation between Hamiltonian curves on $\{H=k\}$ and $\MP$-geodesics on $\{E=k\}$. 

\begin{lemma} \label{lemma:ham_eqs_are_MP_geo}
Let $(g,\alpha,U)$ be an $\MP$-system on the simple manifold $M$. Let $H$ be the associated Hamiltonian. If $(x,\xi) \in \{H=k\}$ is a Hamiltonian curve for $H$, then $\dot{x}$ is $\MP$-geodesic on and $(x,\dot{x}) \in \{E=k\}$.
\end{lemma}

\begin{proof}
This is part of Lemma \ref{lemma:mp_flow}. We only prove the part of the energy. Let $(x(s),\xi(s))$ be a Hamiltonian curve of $H$ on the energy level $\{H=k\}$. From the Hamilton equations we have $\dot{x}^{k}=g^{ik}(\xi_{i}+\alpha_{i})$. Then, 
\begin{align*}
    k&=H(x(s),\xi(s)) \\
    &=\frac{1}{2}|\xi_{i}+\alpha_{i}|^{2}+U \\
    &=\frac{1}{2}g^{ij} (g_{ik}\dot{x}^{k}-\alpha_{i}+\alpha_{i})(g_{j\ell}\dot{x}^{\ell}-\alpha_{\ell}+\alpha_{\ell})+U(x) \\
    &=\frac{1}{2}g^{ij}g_{ik}g_{j\ell}\dot{x}^{k}\dot{x}^{\ell}+U(x) \\
    &=\frac{1}{2}g_{ik}\dot{x}^{k}\dot{x}^{i}+U(x) \\
    &=E(x(s),\dot{x}(s)).
\end{align*}
This finishes the proof.
\end{proof}

See Lemma \ref{lemma:mp_flow} for several interpretations of the $\MP$-flow from a symplectic point of view.

The following result is motivated (and follows) by the fact that multiplying a Hamiltonian by an elliptic factor does not change the Hamiltonian curves (up to reparanetrization) at zero energy level. It gives a relation between Hamiltonian curves for $\tilde{H}=\tilde{H}_{\mu,k}$ and $H$. The proof is based on \cite{plamen}*{Section~3.2}, see also \cite{losu}.

\begin{prop}
Let $(\tilde{x}(s),\tilde{\xi}(s))$ be a Hamiltonian curve of $\tilde{H}$ on the energy level $\{\tilde{H}=k\}$. Then, $(\tilde{x}(s),{\tilde{\xi}}(s))$ is a reparametrization of a Hamiltonian curve for $H$ on the energy level $\{H=k\}$.    
\end{prop}

\begin{proof}
Because of the form of $\tilde{H}=\Tilde{H}_{\mu,k}$, the curve $(\tilde{x}(s),\tilde{\xi}(s))$ satisfy the following equations 
\begin{align*}
    \dot{\tilde{x}}^{k}&=\mu(\tilde{x}) \partial_{\tilde{\xi}_{k}}(H(\tilde{x},\tilde{\xi})), \\
\dot{\tilde{\xi}}_{k}&=-(\partial_{\tilde{x}^{k}}\mu(\tilde{x})) (H(\tilde{x},\tilde{\xi})-k)-\mu(\tilde{x}) \partial_{\tilde{x}^{k}}(H(\tilde{x},\tilde{\xi})).
\end{align*}
Since $(\tilde{x},\tilde{\xi}) \in \{H=k\}=\{\tilde{H}=k\}$, then the equations become
\begin{align*}
   \dot{\tilde{x}}^{k}&=\mu(\tilde{x}) \partial_{\tilde{\xi}_{k}}(H(\tilde{x},\tilde{\xi})), \\
\dot{\tilde{\xi}}_{k}&=-\mu(\tilde{x}) \partial_{\tilde{x}^{k}}(H(\tilde{x},\tilde{\xi})). 
\end{align*}
Let $(x(s),\xi(s))$ be a Hamiltonian curve of $H$ on the energy level $\{H=k\}$. Then, $(x(s),\xi(s))$ solves
\begin{align*}
\dot{x}^{k}&=\partial_{\xi_{k}}(H(x,\xi)), \\
\dot{\xi}_{k}&=- \partial_{x^{k}}(H(x,\xi)).    
\end{align*}
Set $z:=(x(0),\xi(0))$. Let $\beta$ to solve 
\[ \dot{\beta}(s)=\mu(x(s)), \quad \beta(0)=0. \]
Then, $(x,\xi) \circ \beta(s)$ satisfy the same equations as $(\tilde{x}(s),\tilde{\xi}(s))$. By uniqueness of solutions for an ODE, we find $(\tilde{x}(s),\tilde{\xi}(s))=(x,\xi) \circ \beta(s)$.

To see that $(x,\xi) \in \{H=k\}$, observe that since $(x,\xi) \circ \beta(s)=(\tilde{x}(s),\tilde{\xi}(s)) \in \{\tilde{H}=k \}$, then 
\[ k=\tilde{H}(\tilde{x}(s),\tilde{\xi}(s))=\frac{\mu(\tilde{x})}{2}|\tilde{\xi}(s)+\alpha(\tilde{x})|_{g(\tilde{x})}^{2}+\mu(\tilde{x})(U(\tilde{x})-k)+k. \]
Since $\mu>0$ , we obtain
\begin{equation} \label{eq:repara}
    k=\frac{1}{2}|\tilde{\xi}(s)+\alpha(\tilde{x})|_{g(\tilde{x})}^{2}+U(\tilde{x}).
\end{equation}
Note that since $\mu>0$, then $\dot{\beta}>0$. So, $\beta$ is increasing and has an inverse function $\beta^{-1}$. Thus, if we compose \eqref{eq:repara} with $\beta^{-1}$, we find $H(x(s),\xi(s))=k$.
\end{proof}

\section{The New Gauge Equivalence} \label{sec:defin}

In this section we define a more general version of gauge rigidity than the one presented in \cite{az}. Under this notion, we are able to generalize theorems about boundary rigidity presented on that work in Section \ref{sec:thms}.

Recall that associated to the $\MP$-system $(g,\alpha,U)$ (with fixed energy $k$), we have to the magnetic system $(2(k-U)G,\alpha)$ (of energy $\frac{1}{2}$). A natural question is the following: are there other $\MP$-systems that reduce to the same magnetic system? The answer to this question is affirmative. Indeed, the system $(\tilde{g},\tilde{\alpha},\tilde{U}):=(g/\mu,\alpha,\mu(U-k)+k)$ (with fixed energy $k$) has the same reduction. Indeed,
\[ 2(k-\tilde{U})\tilde{g}=2(k- [\mu(U-k)+k])\frac{g}{\mu}=2(k-U)g. \]

As was noticed in Section \ref{sec:hamiltonians}, another motivation to consider systems of this form is because of the Hamiltonians. Recall that we can ``recover'' the Hamiltonian of the reduction $H_{r}$ by looking at a particular case of $\tilde{H}=\tilde{H}_{\mu,k}$.

Finally, we would like to mention that in \cite{az}*{Section 6}, there is a comment about a possible ``weak'' notion of gauge equivalence, in which the authors suggest the a similar notion that they used, up to the conformal factor $(k-U')^{-1}(k-f^{*}U)$.

Motivated by these facts, we propose the following definition.

\begin{defin} \label{defin:gauge}
We say that two $\mathcal{M P}$-systems $(g, \alpha, U)$ and $(g', \alpha', U')$ are $k$-\emph{gauge equivalent} if there is a diffeomorphism $f\colon M \to M$ with $f|_{\partial M}=id_{\partial M}$, a smooth function $\varphi \colon M \to \R$ with $\varphi|_{\partial M}=0$, and a strictly positive function $\mu \in C^{\infty}(M)$, such that $g'=\frac{1}{\mu}f^{*}g$, $\alpha'=f^* \alpha+d \varphi$ and $U'=\mu (f^{*}U-k)+k$.
\end{defin}

\begin{rmk} \hfill
\begin{enumerate}
    \item Take $\mu\equiv 1$ in the previous definition to obtain the notion of gauge equivalence given in \cite{az}. 
    \item This definition allows us to obtain a magnetic system (i.e., $U=0$) be equivalent to an $\MP$-system, because the new potential should be $k(1-\mu)$. To obtain the usual definition of gauge equivalence for magnetic systems (see \cite{dpsu}), we should impose $\mu \equiv 1$ and $U=U'=0$.
    \item The level of energy fixed at the beginning of the definition is used to ensure that we obtain a group of gauge transformations. See below for the proof of this fact.
    \item Note that $k>\max_{x \in M}U$ if and only if $k>\max_{x \in M} [\mu(f^{*}U-k)+k]$. This is crucial when working with the actions as was mentioned above, see \cite{az}.
    \item Two $k$-gauge equivalent have the same boundary action function of energy $k$ and something similar holds for the scattering relation, see Lemma \ref{lemma:mp_data}.
\end{enumerate}
\end{rmk}

Now we show that the set of transformations that we just defined is a group. Consider an $\MP$-system $(g,\alpha,U)$. Then $(\frac{1}{\mu_{1}}f_{1}^{*}g,f_{1}^{*}\alpha+d\varphi_{1},\mu_{1} (f_{1}^{*}U-k)+k)$ is $k$-gauge equivalent to $(g,\alpha,U)$, where $f_{1}\colon M \to M$ is a diffeomorphism with $f_{1}|_{\partial M}=id_{\partial M}$, $\varphi_{1} \colon M \to \R$ is a smooth function with $\varphi_{1}|_{\partial M}=0$, and $\mu_{1} \in C^{\infty}(M)$ is a strictly positive function. Take another diffeomorphism $f_{2} \colon M \to M$ with $f_{2}|_{\partial M}=id_{\partial M}$, another smooth function $\varphi_{2} \colon M\to \R$ with $\varphi_{2}|_{\partial M}=0$, another strictly positive function $\mu_{2} \in C^{\infty}(M)$, and apply it to the new system to obtain the following metric, 1-form, and potential
\[ 
\frac{1}{\mu_{2} \cdot f_{2}^{*} \mu_{1}} (f_{1} \circ f_{2})^{*}g, \qquad (f_{1} \circ f_{2})^{*}\alpha+d(f_{2}^{*}\varphi_{1}+\varphi_{2}), \qquad (\mu_{2} \cdot f_{2}^{*}\mu_{1}) ((f_{2} \circ f_{1})^{*}U-k )+k.  
\]
So, we obtain again a system that is $k$-gauge equivalent to the initial one. Indeed, the diffeomorphism is $f_{1} \circ f_{2}$ and satisfies $f_{1} \circ f_{2}|_{\partial M}=id_{\partial M}$ because each of the $f_{i}$'s satisfies this condition. The smooth function is $f_{2}^{*}\varphi_{1}+\varphi_{1}$, which vanishes on the boundary because each $\varphi_{i}$ vanishes there, and $f_{2}|_{\partial M}=id_{\partial M}$. Finally, the elliptic factor is $\mu_{2} \cdot f_{2}^{*}\mu_{1}$, which is strictly positive because each $\mu_{i}$ is strictly positive as well.

To show that $k$-gauge equivalent have the same data, we need a preliminary result, which will also be used in the proof of the rigidity theorems in Section \ref{sec:thms}.

\begin{lemma} \label{lemma:mp_red}
Let $(g,\alpha,U)$ and $(g',\alpha',U')$ be two $k$-gauge equivalent $\MP$-system. Then, their reduced magnetic systems are magnetic gauge equivalent at energy $\frac{1}{2}$. Reciprocally, if $(G,\alpha)$ and $(G,\alpha')$ are gauge equivalent magnetic systems (of energy $\frac{1}{2}$) given by the reduction of the $\MP$-systems $(g,\alpha,U)$ and $(g',\alpha',U')$, then the $\MP$-systems are $k$-gauge equivalent.
\end{lemma}

\begin{proof}
The reduced magnetic systems are 
\[ (2(k-U)g,\alpha), \quad (2(k-f^{*}U)f^{*}g,f^{*}\alpha+d\varphi), \]
where $f \colon M \to M$ is a diffeomorphism which is the identity on the boundary and $\varphi \in C^{\infty}(M)$ with $\varphi|_{\partial M}=0$. These systems are magnetic gauge equivalent by definition. To prove the condition about the energy, note that if $\sigma_{1}$ is an $\MP$-geodesic for $(g,\alpha,U)$ and $\sigma_{2}$ is an $\MP$-geodesic for $(g',\alpha',U')$, then by Lemma \ref{lemma:basics_mp} we have that their reparametrizations $\gamma_{1}(s_{1})=\sigma(t_{1}(s_{1}))$ and $\gamma_{2}(s_{2})=\sigma(t_{2}(s_{2}))$ are magnetic geodesics of the corresponding reduced magnetic systems. Then, 
\[2(k-U) \frac{g_{ij}}{2} \frac{d\gamma_{1}^{i}}{ds_{1}}\frac{d\gamma_{1}^{j}}{ds_{1}}=2(k-U) \frac{g_{ij}}{2} \frac{d\sigma_{1}^{i}}{dt_{1}}\frac{d\sigma_{1}^{j}}{dt_{1}} \left(\frac{dt_{1}}{ds_{1}}\right)^{2}=2(k-U)^{2}\left(\frac{dt_{1}}{ds_{1}}\right)^{2}=\frac{1}{2}.\]
Similarly, 
\begin{align*}
    2(k-f^{*}U) \frac{(f^{*}g)_{ij}}{2} \frac{d\gamma_{2}^{i}}{ds_{2}}\frac{d\gamma_{2}^{j}}{ds_{2}}&=2(k-U^{*}f) \frac{(f^{*}g)_{ij}}{2} \frac{d\sigma_{2}^{i}}{dt_{2}}\frac{d\sigma_{2}^{j}}{dt_{2}} \left(\frac{dt_{2}}{ds_{2}}\right)^{2} \\
    &=2\mu^{2}(k-f^{*}U)^{2}\left(\frac{dt_{2}}{ds_{2}}\right)^{2} \\
    &=\frac{1}{2}.
\end{align*}

To prove the second part, recall that if $(G,\alpha)$ and $(G',\alpha')$ are gauge equivalent, then there exists a diffeomorphism $f \colon M \to M$ fixing the boundary and a smooth function $\varphi$ vanishing in the boundary so that $G'=f^{*}G$ and $\alpha'=f^{*}\alpha+d\varphi$. Note that the condition about the metric imply
\[ 2(k-U')g'=G'=f^{*}G=2(k-f^{*}U)f^{*}g. \]
Define $\mu=\frac{k-U'}{k-f^{*}U}$. Then 
\[ g'=\frac{1}{\mu}f^{*}g, \qquad\alpha'=f^{*}\alpha+d\varphi, \qquad U'=\mu(f^{*}U-k)+k, \]
which shows that $(g,\alpha,U)$ and $(g',\alpha',U')$ are $k$-gauge equivalent.
\end{proof}

Using this, we show that two $k$-gauge equivalent $\MP$-systems have the same data.

\begin{lemma} \label{lemma:mp_data}
Let $(g,\alpha,U)$ and $(g',\alpha',U')$ be two $k$-gauge equivalent $\MP$-systems. Then they boundary action functions at energy level $k$ coincide. Furthermore, if $g|_{\partial M}=g'|_{\partial M}$, $U|_{\partial M}=U'|_{\partial M}$ and $i^{*}\alpha=i^{*}\alpha'$, where $i \colon \partial M \to M$ is the embedding map, then the scattering relations at energy level $k$ coincide.
\end{lemma}

\begin{proof}
By Lemma \ref{lemma:mp_red}, we have that the reduced magnetic systems are gauge equivalent. Since gauge equivalent magnetic systems have the same boundary action function, the first part follows from Lemma \ref{lemma:basics_mp}.

The second part follows directly from the first one and \cite{az}*{Theorem 4.3}.
\end{proof}

To finish this section, let us discuss the counterexample presented in \cite{az}*{Section 3} and how is related with the new definition. There, the authors consider the $\MP$-systems $(\tilde{g},\alpha,\varphi)$ and $(\tilde{g}',\alpha,2\psi)$, where $1 \leq \varphi <\frac{3}{2}$ in $M^{\text {int }}$ and $\varphi \equiv \frac{3}{2}$ on $\partial M$, $\frac{3}{4}<\psi \leq 1$ in $M^{\text {int }}$ and $\psi \equiv \frac{3}{4}$ on $\partial M$, $\tilde{g}=\frac{1}{2(3-\varphi)}g$, $\tilde{g}'=\frac{1}{2(3-2\psi)}g$, where $g$ is a Riemannian metric in $M$, and $\alpha$ is a 1-form in $M$. These systems have the same boundary action function for energy $k=3$, but they are no gauge equivalent in the sense discussed there, i.e., there is no diffeomorphism $f \colon M \to M$ with $2\psi=\varphi \circ f$. The last claim is clearly true because $\varphi<\frac{3}{2}<2\psi$ in $M^{\text {int }}$. However, $(\tilde{g},\alpha,\varphi)$ and $(\tilde{g}',\alpha,2\psi)$ are $k$-gauge equivalent (with $k=3$). Indeed, take $\mu=\frac{(3-2\psi)}{(3-\varphi)}$. Then
\[ \frac{1}{\mu}\tilde{g}=\frac{1}{2(3-2\psi)}g=\tilde{g}'. \]
Also, 
\[ \mu(\varphi-3)+3=(2\psi-3)+3=2\psi. \]
This shows our claim. 

There are two takeaways here. First, the counter-example given by Y. Assylbekov and H. Zhou is not a counter-example anymore if we work with the definition of $k$-gauge equivalence. Secondly, the restrictions that appeared on the potential while working with $\MP$-systems with one level of energy disappear when we work with our new definition of gage, that is, we are not trying to obtain $U=f^{*}U$ and $g=f^{*}g$ using one energy level, but a more general relation. Indeed, we will see that we are able to obtain results on boundary rigidity with only one level of energy up to $k$-gauge equivalent systems, see Section \ref{sec:thms}.

\section{Rigidity Theorems} \label{sec:thms}

In this section we obtain results about rigidity of $\MP$-systems proved working on one energy level. Our proof follow the ideas of the rigidity theorems proved in \cite{az}. We reduce the $\MP$-systems to apply the rigidity results in \cite{dpsu}.

\subsection{Boundary rigidity}

We obtain the following boundary rigidity result in a conformal class, cf. \cite{az}*{Theorem 5.1}.

\begin{thm} \label{thm:bound_rid_conf_class}
Let $(g, \alpha, U)$ and $(g', \alpha', U')$ be simple $\MP$-systems on $M$ with the same boundary action functions for energy $k$. If $g'$ is conformal to $g$, that is, if there exists a smooth positive function $\mu$ such that $g'=\frac{1}{\mu}g$, then $U'=\mu(U-k)+k$ and there exist a function $\varphi \in C^{\infty}(M)$ with $\varphi|_{\partial M}=0$, such that, $\alpha'=\alpha+d\varphi$. In particular, $(g', \alpha', U')$ is $k$-gauge equivalent to $(g, \alpha, U)$.    
\end{thm}

To prove this result, we will need the corresponding result for magnetic systems:

\begin{thm}[\cite{dpsu}*{Theorem 6.1}] \label{thm:bound_rid_conf_class_mag}
Let $(g, \alpha)$ and $(g', \alpha')$ be simple magnetic systems on $M$ whose boundary action functions $\mathbb{A}|_{\partial M \times \partial M}$ and $\mathbb{A}'|_{\partial M \times \partial M}$ coincide. If $g'$ is conformal to $g$, then $g=g'$ and $\alpha'=\alpha+dh$ for some smooth function $h$ on $M$ vanishing on $\partial M$.    
\end{thm}

\begin{proof}[Proof of Theorem \ref{thm:bound_rid_conf_class}]
Let $G=2(k-U) g$, $G'=2(k-U') g'$ with Ma\~n\'e action potentials of energy $\frac{1}{2}$ given by $\mathbb{A}_{G}$ and $\mathbb{A}_{G'}$, respectively. By Lemma \ref{lemma:basics_mp}, these are simple magnetic systems and $\A_{G}|_{\partial M \times \partial M}=\A_{G}|_{\partial M \times \partial M}$. 

Since $g=\frac{1}{\mu}g$, we see that
\[
G'=2(k-U')g'=2(k-U')\frac{1}{\mu} g=\frac{1}{\mu}\frac{k-U'}{k-U}G.
\]
Applying Theorem \ref{thm:bound_rid_conf_class_mag}, we get $G'=G$, i.e.
\[ \frac{k-U'}{k-U} \equiv \mu,\]
and that there is $\varphi \in C^{\infty}(M)$, with $\varphi|_{\partial M}=0$, such that $\alpha'=\alpha+d \varphi$. Hence, 
\[ g'=\frac{1}{\mu}g, \qquad\alpha'=\alpha+d\varphi, \qquad U'=\mu(U-k)+k, \]
that is, $(g,\alpha,U)$ is $k$-gauge equivalent to $(g',\alpha',U')$.
\end{proof}

Before presenting the result for analytic manifolds, let us mention some conventions. As in \cite{dpsu}, we work on a real analytic manifold $M$, with smooth boundary $\partial M$ that does not need to be analytic. We say that an object (2-tensor, 1-tensor or function) is analytic in the set $X$ (not necessarily open), if it is analytic in a neighborhood of $X$.

\begin{thm} \label{thm:bound_rid_an}
If $M$ is a real-analytic compact manifold with boundary, $(g, \alpha, U)$ and $(g', \alpha', U')$ are simple real-analytic $\MP$-systems on $M$ with the same boundary action functions for energy $k$, then these systems are $k$-gauge equivalent.    
\end{thm}

This is a generalization of the following result that will be used in the proof.

\begin{thm}[\cite{dpsu}*{Theorem 6.2}] \label{thm:bound_rid_an_mag}
If $M$ is a real-analytic compact manifold with boundary, and $(g, \alpha)$ and $(g', \alpha')$ are simple real-analytic magnetic systems on $M$ with the same boundary action function, then these systems are gauge equivalent.    
\end{thm}

\begin{proof}[Proof of Theorem \ref{thm:bound_rid_an}]
Let $G=2(k-U) g$, $G'=2(k-U') g'$ with Ma\~n\'e action potentials of energy $\frac{1}{2}$ given by $\mathbb{A}_{G}$ and $\mathbb{A}_{G'}$, respectively. By Lemma \ref{lemma:basics_mp}, these are simple real-analytic magnetic systems and $\A_{G}|_{\partial M \times \partial M}=\A_{G}|_{\partial M \times \partial M}$. 

In virtue of Theorem \ref{thm:bound_rid_an_mag}, the magnetic systems $(G, \alpha)$ and $(G', \alpha')$ are gauge equivalent. Hence, by Lemma \ref{lemma:mp_red}, we obtain the desire result.
\end{proof}

Finally, in dimension 2 we obtain boundary rigidity for $\MP$-systems up to $k$-gauge equivalence. This generalizes the celebrated result \cite{pu2005}*{Theorem 1.1}.

\begin{thm} \label{thm:bound_rid_dim2}
If $\dim M=2$ and $(g, \alpha, U)$ and $(g', \alpha', U')$ are simple $\MP$-systems on $M$ with the same boundary action functions of energy $k$, then these systems are $k$-gauge equivalent.    
\end{thm}

As before, this theorem also generalizes the result for magnetic systems, which will be useful for the proof

\begin{thm}[\cite{dpsu}*{Theorem 7.1}] \label{thm:bound_rid_dim2_mag}
If $\dim M=2$ and $(g, \alpha)$ and $(g', \alpha')$ are simple magnetic systems on $M$ with the same boundary action function, then these systems are gauge equivalent.    
\end{thm}

\begin{proof}[Proof of Theorem \ref{thm:bound_rid_dim2}]
Let $G=2(k-U) g$, $G'=2(k-U') g'$ with Ma\~n\'e action potentials of energy $\frac{1}{2}$ given by $\mathbb{A}_{G}$ and $\mathbb{A}_{G'}$, respectively. By Lemma \ref{lemma:basics_mp}, these are simple real-analytic magnetic systems and $\A_{G}|_{\partial M \times \partial M}=\A_{G}|_{\partial M \times \partial M}$.  In light of Theorem \ref{thm:bound_rid_dim2_mag}, the magnetic systems $(G,\alpha)$ and $(G',\alpha')$ are gauge equivalent. Now, Lemma \ref{lemma:mp_red} implies that the $\MP$-systems $(g,\alpha,U)$ and $(g',\alpha',U')$ are $k$-gauge equivalent, finishing the proof. 
\end{proof}

\begin{rmk}
    We can relax the simplicity hypothesis in the theorems of this section. Let $M$ be a compact smooth manifold with boundary, Take $(g,\alpha,U)$ be a simple $\MP$-system on $M$, and let $(g',\alpha',U)$ be a $\MP$-system on $M$. Then, $(2(k-U)g,\alpha)$ is a magnetic simple system in virtue of Lemma \ref{lemma:basics_mp}. Assume that $2(k-U)g|_{\partial M}=2(k-U')g'|_{\partial M}$ and $i^{*}\alpha=i^{*}\alpha'$, where $i\colon \partial M \to M$ is the embedding map. If the restricted scattering relation (see \cite{dpsu}*{Definition 2.4}) of the reduced magnetic system, $\mathfrak{s}$, $\mathfrak{s}'$, are the same, then \cite{herrerosmag}*{Theorem 4.1} implies that $(2(k-U')g',\alpha')$ is also a simple magnetic system. Hence, we can invoke Lemma \ref{lemma:basics_mp} to conclude that $(g',\alpha',U')$ is a simple $\MP$-system on $M$.
\end{rmk}

\subsection{Scattering rigidity}

It is known that for simple $\MP$-systems, the boundary rigidity problem is equivalent to the problem of recovering the data $(g,\alpha,U)$ from knowing of the scattering relation, see \cite{az}*{Section~4.2}. So, we can translate Theorems \ref{thm:bound_rid_conf_class}, \ref{thm:bound_rid_an}, and \ref{thm:bound_rid_dim2} to obtain scattering rigidity results, which we now state.

\begin{thm}
Let $(g, \alpha, U)$ and $(g', \alpha', U')$ be simple $\MP$-systems on $M$ with the same scattering relation for energy $k$. If $g'$ is conformal to $g$, then there exist a strictly positive function $\mu \in C^{\infty}$, and a function $\varphi \in C^{\infty}(M)$ with $\varphi|_{\partial M}=0$, such that $g'=\frac{1}{\mu}g$, $\alpha'=\alpha+d\varphi$, and $U'=\mu(U-k)+k$. In particular, $(g', \alpha', U')$ is $k$-gauge equivalent to $(g, \alpha, U)$.      
\end{thm}

\begin{thm}
If $M$ is a real-analytic compact manifold with boundary, $(g, \alpha, U)$ and $(g', \alpha', U')$ are simple real-analytic $\MP$-systems on $M$ with the scattering relation for energy $k$, then these systems are $k$-gauge equivalent.        
\end{thm}

\begin{thm}
If $\dim M=2$ and $(g, \alpha, U)$ and $(g', \alpha', U')$ are simple $\MP$-systems on $M$ with the same scattering relation of energy $k$, then these systems are $k$-gauge equivalent.        
\end{thm}

Note that to make sense of the condition about the equality of the scattering relation, one should have $\{E=k\}=\{E'=k\}$, where $E$ and $E'$ are the energies for $(g, \alpha, U)$ and $(g', \alpha', U')$, respectively. A reasonable condition that ensures the equality of the level sets is that $g|_{\partial M}=g'|_{\partial M}$, $U|_{\partial M}=U'|_{\partial M}$ and $i^{*}\alpha=i^{*}\alpha'$, where $i\colon \partial M \to M$ is the embedding map.

\section{Final remarks} \label{sec:rmk}

As was commented in the introduction, there are results on boundary rigidity and lens rigidity working with the foliation condition \cite{suv}, \cite{zhou18}. In would be interesting to know if the results here and in \cite{dpsu} can be obtained under that condition instead of the simplicity assumption.

Another possible direction of work is to explore if one can obtain generic boundary rigidity results as in \cite{dpsu}*{Theorem 6.5}. This could require a study of the linearized problem. However, one could also try a more direct method using the relations between magnetic and $\MP$-systems.

Finally, one could try to apply the main tool that we used, that is, the relation between the $\MP$-system $(g,\alpha,U)$ with the magnetic system $(2(k-U),\alpha)$ to generalize results involving magnetic systems. For example, it would be interesting to study the lens rigidity problem for $\MP$-systems (see \cite{herrerosmag}, \cite{zhou18}), or obtaining a reconstruction procedure for the $\MP$-system as for magnetic systems (\cite{du}), study the marked length spectrum rigidity problem for $\MP$-systems (see \cite{mr23}), and study inverse source problems for domains with a non-Euclidean metric, generalizing \cite{lz}.

\appendix

\section{The MP-flow}

In this appendix we summarize some ways of understand the $\MP$-flow. The results comes from the relation between Hamiltonian and Lagrangian mechanics on symplectic manifolds, see \cite{am}*{Chapter 3}.

First we introduce some notation. Let $\pi \colon TM \to M$ and $p \colon T^{*}M \to M$ be the base point projections. Let $\omega_{c}$ be the canonical symplectic form on $T^{*}M$. If we use coordinates $(x,\xi)$ on the cotangent bundle, locally we have
\[ \omega_{c}=\sum_{i=1}^{n} dx^{i} \wedge d\xi^{i}. \]
By using the musical isomorphism $\flat \colon TM \to T^{*}M$, we can pull back $\omega_{c}$ to obtain the canonical symplectic form $\omega_{0}$ on $TM$, given locally by
\begin{align*}
    \omega_{0} &=\sum_{i,j=1}^{n} \left \lbrace  \left(\sum_{k=1}^{n} v^{k} \partial_{x^{j}}g_{ik}  \right) dx^{i} \wedge dx^{j}+g_{ij}dx^{i}\wedge dv^{j}  \right \rbrace,
\end{align*}
see \cite{am}*{Theorem 3.2.13}. We will also consider the twisted versions of these forms. Let $\Omega$ be a 2-form on $M$, then we can pull-it back to a 2-form on $TM$ or $T^{*}M$ by considering $\pi^{*}\Omega$ and $p^{*}\Omega$, respectively. We can twist the symplectic forms on $TM$ and $T^{*}M$ to obtain
\[ \omega:=\omega_{0}+\pi^{*}\Omega, \quad \omega^{*}:=\omega_{c}+p^{*}\Omega. \]
The $\Omega$ that we will consider is the closed 2-form that symbolizes the magnetic part of the system. 

\begin{lemma} \label{lemma:mp_flow} \hfill
\begin{enumerate}
    \item When $\Omega=d\alpha$, the $\MP$-flow is the Euler--Lagrange flow corresponding to the Lagrangian $L(x,v)=\frac{1}{2}|v|_{g}^{2}-\alpha_{x}(v)-U(x)$. Furthermore, the $\MP$-geodesics satisfy the Euler--Lagrange equations.
    \item The $\MP$-flow is the Hamiltonian flow on $(TM,\omega,E)$. Furthermore, the projected trajectories of the flow to $M$ satisfy the equation of $\MP$-geodesics.
    \item The Hamiltonian flow on $(T^{*}M,\omega^{*},\frac{1}{2}|\xi|_{g}^{2}+U)$ identified with a flow on $TM$ by the musical isomorphism, coincides with the $\MP$-flow. Furthermore, the projected trajectories of the flow on $TM$ to $M$ satisfy the equation of $\MP$-geodesics.
    \item When $\Omega=d\alpha$, the Hamiltonian flow on $(T^{*}M,\omega_{c},H)$ identified with a flow on $TM$ by the musical isomorphism, coincides with the $\MP$-flow. Furthermore, the projected trajectories of the flow on $TM$ to $M$ satisfy the equation of $\MP$-geodesics.
\end{enumerate}
\end{lemma}

\begin{proof} \hfill
    \begin{enumerate}
    \item First, we observe that
    \[ \frac{\partial L}{\partial x^{k}}=\frac{1}{2} (\partial_{x^{k}}g_{ij})v^{i}v^{j}-(\partial_{x^{k}}\alpha_{i})v^{i}-\partial_{x^{k}}U. \]
    On the other hand,
    \[ \frac{\partial L}{\partial v^{k}}=g_{ik}v^{i}-\alpha_{k}. \]
    So,
    \[ \frac{d}{dt} \left( \frac{\partial L}{\partial v^{k}}(\sigma,\dot{\sigma}) \right)=\frac{1}{2}(\partial_{x^{\ell}}g_{ik})\dot{\sigma}^{\ell}\dot{\sigma}^{i}+\frac{1}{2}(\partial_{x^{i}} g_{\ell k}) \dot{\sigma}^{i} \dot{\sigma}^{\ell}+g_{ik} \ddot{\sigma}^{i}-(\partial_{x^{\ell}}\alpha_{k})\dot{\sigma}^{\ell}. \]
    Putting all together and relabeling the indices we have
    \[ g_{ik}\ddot{\sigma}^{i}+\frac{1}{2}(\partial_{x^{\ell}}g_{ik}+\partial_{x^{i}}g_{\ell k}-\partial_{x^{k}}g_{i\ell} )\dot{\sigma}^{i}\dot{\sigma}^{\ell}=(\partial_{x^{i}}\alpha_{k}-\partial_{x^{k}}\alpha_{i})\dot{\sigma}^{i}-\partial_{x^{k}}U. \]
    Multiplying by $g^{jk}$ and using that $g^{jk}(\partial_{x^{i}}\alpha_{k}-\partial_{x^{k}}\alpha_{i})=g^{jk}\Omega_{ik}=Y_{i}^{j}$, we find
    \[ \ddot{\sigma}^{j}+\Gamma_{i\ell}^{j}\dot{\sigma}^{i}\dot{\sigma}^{\ell}=Y_{i}^{j}\dot{\sigma}^{i}-g^{jk}\partial_{x^{k}}U, \]
    which is equation \eqref{eq:mp-geo}.
    \item First we compute the Hamiltonian vector field $X_{E}=A^{i}\partial_{x^{i}}+B^{i}\partial_{v^{i}}$ associated to the system. Let $V=C^{i}\partial_{x^{i}}+D^{i}\partial_{v^{i}}$. We compute and obtain
    \begin{align*}
        dE(V) &=\left( \frac{1}{2}(\partial_{x^{k}} g_{ij})v^{i}v^{j}+\partial_{x^{k}}U \right)C^{k}+g_{ik}v^{i}D^{k}, \\
        \omega(X_{H},V)&=g_{ij}A^{i}D^{j}-g_{ij}C^{i}B^{j}+\Omega_{ij}A^{i}C^{j}.
    \end{align*}
    From here, we obtain $A^{i}=v^{i}$ and 
    \[ B^{i}=-\Gamma_{jk}^{i}v^{j}v^{k}+Y_{j}^{i}v^{j}-g^{ij}\partial_{x^{j}}U, \]
    where we used $\Gamma_{jk}^{i}=\frac{1}{2}g^{i\ell}(\partial_{x^{\ell}}g_{jk})$ and $Y_{j}^{i}=g^{ik}\Omega_{kj}$. Hence, the Hamiltonian vector field is
    \[ G_{\MP}:=X_{H}=v^{i}\partial_{x^{i}}+(-\Gamma_{jk}^{i}v^{j}v^{k}+Y_{j}^{i}v^{j}-g^{ij}\partial_{x^{j}}U)\partial_{v^{i}}. \]
    Hence, Hamilton equations are
    \begin{align} \label{eq:hameq_mus1}
        \dot{x}^{i}&=v^{i}, \\ \label{eq:hameq_mus2}
        \dot{v}^{i}&=-\Gamma_{jk}^{i}v^{j}v^{k}+Y_{j}^{i}v^{j}-g^{ij}\partial_{x^{j}}U,
    \end{align}
    which are the equations that define the $\MP$-flow. From here, it is clear that $x$ satisfies \eqref{eq:mp-geo}.
    \item Since this time we are working with the twisted form $\omega_{c}+\pi^{*}\Omega$, the Hamiltonian vector field now is 
    \[ g^{ik}\xi_{k}\partial_{x^{i}}+\left( -\frac{1}{2}(\partial_{x^{i}}g^{jk})\xi_{j}\xi_{k}-\partial_{x^{i}}U+\Omega_{ij}g^{ik}\xi_{k} \right)\partial_{\xi_{i}}. \]
    Then, the Hamilton equations are
    \begin{align} \label{eq:hameq1}
        \dot{x}^{i}&=g^{ik}\xi_{k}, \\ \label{eq:hameq2}
        \dot{\xi}_{i}&=-\frac{1}{2} (\partial_{x^{i}}g^{jk})\xi_{j}\xi_{k}-\partial_{x^{i}}U+\Omega_{ij}g^{jk}\xi_{k}.
    \end{align}
    Letting $v^{i}=g^{ik}\xi_{k}$, we see that $\dot{x}^{i}=v^{i}$ and $\dot{v}^{i}=-g^{im}\partial_{x^{\ell}}g_{mj}v^{\ell}v^{j}+g^{ik}\dot{\xi}_{k}$. Hence, under the musical isomorphism we have that the equations \eqref{eq:hameq1} and \eqref{eq:hameq2} take the form of \eqref{eq:hameq_mus1} and \eqref{eq:hameq_mus2}, respectively. As before, from here we see that $x$ satisfies \eqref{eq:mp-geo}.
    \item Let 
    \begin{align*}
        \mathcal{L} \colon TM & \to T^{*}M, \\
        (x,v) & \mapsto \left(x, \frac{\partial L}{\partial v}(x,v) \right),
    \end{align*}
    be the Legendre transform. Since $\mathcal{L}(x,v)=(x,v^{\flat}-\alpha)$, then $\mathcal{L}^{-1}(x,p)=(x,\alpha^{\sharp}+\xi^{\sharp})$. Note that $H=E \circ \mathcal{L}^{-1}$, and that $L$ is convex and superlinear (see \cite{ci}, \cite{paternain}). Hence, the result follows directly from \cite{ci}*{Proposition 1-4.1} (cf. \cite{arnoldclas}*{Section 15}). 
    
    We also prove the result directly. Since we are working with the canonical form on $T^{*}M$, the Hamiltonian vector field is given by $X_{H}=\partial_{\xi_{i}}H\partial_{x^{i}} -\partial_{x^{i}}H\partial_{\xi_{i}}$. Then, the Hamilton equations are
    \begin{align*}
        \dot{x}^{k}&=g^{ik}(\xi_{i}+\alpha_{i}), \\
        \dot{\xi}_{k}&=-\frac{1}{2}(\partial_{x^{k}}g^{ij})(\xi_{i}+\alpha_{i})(\xi_{j}+\alpha_{j})-g^{ij}(\partial_{x^{k}}\alpha_{i})(\xi_{j}+\alpha_{j})-\partial_{x^{k}}U.
    \end{align*}
    Let $\eta=\xi+\alpha(x)$. Then, $\dot{\eta}_{k}=\dot{\xi}_{k}+(\partial_{x^{i}}\alpha_{k})\dot{x}^{i}$ and $\dot{x}^{k}=g^{ik}\eta_{i}$. Hence, Hamilton equations take the form
    \begin{align*}
        \dot{x}^{k}&=g^{ik}\eta_{i}, \\
        \dot{\eta}_{k}&=-\frac{1}{2}(\partial_{x^{k}}g^{ij})\eta_{i}\eta_{j}-g^{ij}(\partial_{x^{k}}\alpha_{i})\eta_{j}+g^{ij}(\partial_{x^{i}}\alpha_{k})\eta_{j}-\partial_{x^{k}}U.
    \end{align*}
    These equations are no other than \eqref{eq:hameq1} and \eqref{eq:hameq2}. Hence, the conclusions follows from the previous part.
\end{enumerate}
\end{proof}

\begin{rmk}
    By using the formula given in \cite{mazzucchelli}*{page 4}, we note that the Hamiltonian vector field $G_{\MP}$ from the second point in Lemma \ref{lemma:mp_flow} coincides with the Euler--Lagrange field from the first point on that lemma. Clearly, $G_{\MP}$ is also the generator of the $\MP$ flow mentioned in \cite{az}.
\end{rmk}

\end{document}